\documentclass{amsart}
\usepackage[english]{babel}
\usepackage[latin1]{inputenc}
\usepackage[dvips,final]{graphics}
\usepackage{amsmath,amsfonts,amssymb,amsthm,amscd,array,stmaryrd,mathrsfs, mathdots, epigraph}
\usepackage{pstricks}
 \usepackage[all]{xy}
 \usepackage{url}
\usepackage{multirow, blkarray}
\usepackage{booktabs}
\usepackage{textcomp}
 \usepackage[final]{epsfig}
 \usepackage{color}
\theoremstyle{plain}
\newtheorem{thm}{Theorem}
\newtheorem{lem}{Lemma}[section]

\newtheorem{prop}[lem]{Proposition}
\newtheorem{defn}[lem]{Definition}

\theoremstyle{definition}
\newtheorem{rem}[lem]{Remark}
\newtheorem{ex}[lem]{Example}

\newcommand{\Z}{\mathbb{Z}}
\newcommand{\C}{\mathbb{C}}

\newcommand{\G}{\mathcal{G}}

\newcommand{\Qc}{\mathcal{Q}} 
\newcommand{\Rc}{\mathcal{R}}

\newcommand{\SL}{\mathrm{SL}}

\newcommand{\OSp}{\mathrm{OSp}}
\newcommand{\SSL}{\mathrm{SL}}

\def\a{\alpha}
\def\b{\beta}
\def\d{\delta}

\def\g{\gamma}

\def\om{\omega}

\begin{document}

\title{Cluster algebras with Grassmann variables}

\author{Valentin Ovsienko}

\address{
Valentin Ovsienko,
CNRS,
Laboratoire de Math\'ematiques 
U.F.R. Sciences Exactes et Naturelles 
Moulin de la Housse - BP 1039 
51687 REIMS cedex 2,
France}
\email{valentin.ovsienko@univ-reims.fr}

\author{Michael Shapiro}

\address{Michael Shapiro,
Department of Mathematics, 
Michigan State University, 
East Lansing, MI 48824,
USA}
\email{machace5@math.msu.edu}


\begin{abstract}
We develop a version of cluster algebra extending the ring of Laurent polynomials
by adding Grassmann variables.
These algebras can be described in terms of ``extended quivers''
which are oriented hypergraphs.
We describe mutations of such objects and define a corresponding commutative superalgebra.
Our construction includes the notion of weighted quivers that has already appeared in different contexts.
This paper is a step towards understanding the notion of cluster superalgebra.
\end{abstract}

\maketitle


\section*{Introduction}

Cluster algebras, discovered by Fomin and Zelevinsky~\cite{FZ1}, 
are a special class of commutative associative algebras.
It was proved by many authors that the coordinate rings of many algebraic varieties
arising in the Lie theory, such as Lie groups of matrices, Grassmannians,
various moduli spaces, etc. have structures of cluster algebra; for a survey, see~\cite{Wil}.
Cluster algebras naturally appear in algebra, geometry and combinatorics, they are also closely related to
integrable systems.

A cluster algebra is a subalgebra of the algebra of
{\it Laurent polynomials} in
$\Z[x^{\pm1}_1,\ldots,x^{\pm1}_n]$
generated by certain polynomials with positive integer coefficients.
A cluster algebra is usually defined with the help of a quiver (an oriented graph)
with no loops and no $2$-cycles;
the generators of the algebra are defined with the help of
{\it exchange relations} and {\it quiver mutations}.

Our goal is to introduce a version of cluster algebra
with nilpotent (Grassmann) variables
$\{\xi_1,\ldots,\xi_m\}$, that anticommute with each other,
and in particular, square to zero.
The algebras we consider are certain subalgebras of the ring
$
\C[x^{\pm1}_1,\ldots,x^{\pm1}_n,\xi_1,\ldots,\xi_m],
$
that are Laurent polynomials in $x_1,\ldots,x_n$.
Unfortunately, we can only treat mutations of even variables
leaving the odd ones frozen.
In this sense, the correct notion of cluster superalgebra is still out of reach.
We believe that the correct notion of mutations of odd variables should extend the
coordinate transformations considered in~\cite{MOT} and~\cite{PZ,IPZ}.

We consider a notion of ``extended quiver'' which is a {\it hypergraph} extending
 a classical quiver.
The main ingredients are modified exchange relations and quiver mutations.
The vertices of the classical quiver are labeled by the even variables,
the new vertices are labeled by the Grassmann variables (diamonds, in the Figure below).
Essentially, the mutations of an extended quiver
can be described by the following diagrams:
$$
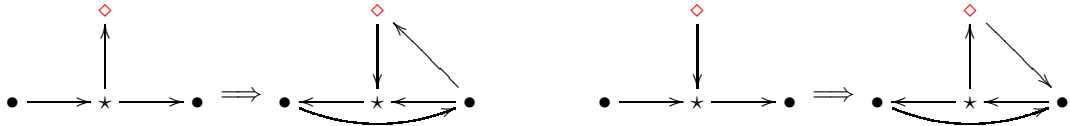
\xymatrix{
&{\color{red}\diamond}\ar@{<-}[d]&\\
\bullet&\star\ar@{<-}[l]&\bullet\ar@{<-}[l]
}
\xymatrix{\\
\Longrightarrow }
 \xymatrix{
&{\color{red}\diamond}\ar@{->}[d]\ar@<1pt>@{<-}[rd]&
\\
\bullet\ar@/^-0.7pc/[rr]&\star\ar@{->}[l]&\bullet\ar@{->}[l]
}
\qquad\qquad
\xymatrix{
&{\color{red}\diamond}\ar@{->}[d]&\\
\bullet&\star\ar@{<-}[l]&\bullet\ar@{<-}[l]
}
\xymatrix{\\
\Longrightarrow}
 \xymatrix{
&{\color{red}\diamond}\ar@{<-}[d]\ar@<1pt>@{->}[rd]&
\\
\bullet\ar@/^-0.7pc/[rr]&\star\ar@{->}[l]&\bullet\ar@{->}[l]
}
$$
The ``underlying quiver'', with vertices shown as black bullets mutates in a standard way.
The additional vertices denoted by red diamonds that represent a group of Grassmann variables connected to a standard (even) vertex,
behave in a way quite different from the standard mutation rules.
Quite remarkably, the above mutation rule includes the notion of {\it weighted quivers} (see~\cite{OT} and references therein)
which is in a sense dual to the Bernstein-Gelfand-Ponomarev functor.
We will explain how to reduce the above mutation rules to the classical ones combined with a
transformation of quivers that we call a ``monomial transform''.

This paper is based on the unpublished preprint~\cite{Ovs}, however
we modify the exchange relations suggested by~\cite{Ovs} in such a way
the restrictions on quiver mutations disappear.
The main motivation for our construction is the idea to develop
a complete notion of {\it cluster superalgebra}.
One concrete example of our general construction is the notions of
{\it superfriezes} considered in~\cite{MOT}.
Let us also mention another attempt to develop the notion of cluster superalgebra~\cite{LMRS}
which is quite different from our construction.
In particular, the exchange relations of~\cite{LMRS}, similarly to~\cite{Ovs},
 are at most quadratic in odd variables.

We would like to pay attention to the fact that unlike the present paper, \cite{LMRS} 
contains expressions for mutations of odd variables. However, in our opinioin the expressions 
in \cite{LMRS} has some flaws the most evident of which is that
the transcendence degree of the cluster algebra is generally speaking not mutationally invariant.

\section{Extended quivers and their mutations}\label{DeFCluSS}

We introduce extended quivers, and describe their mutation rules.
It turns out that an extended quiver is not a graph but
an oriented {\it hypergraph}.
More precisely, given a quiver $\Qc$,
we add new, odd (or colored), vertices, and complete
the set of edges of $\Qc$ by adding
some $2$-paths joining three vertices.

The reason for this notion is the general idea of superalgebra and supergeometry,
that supersymmetric version of every object should be understood as its ``square root''.
The notion of extended quiver is an attempt to apply this idea in combinatorics:
a square root of an edge in a graph is understood as a $2$-path joining two
odd vertices through an even vertex.

\subsection{Introducing extended quivers}\label{IntroQuiS}

Let us recall that a {\it hypergraph}  is a generalization of a graph in the following sense.
A hypergraph~$\G$ is a pair~$(\G_0,\G_1)$, where~$\G_0$ is the set of vertices and
$\G_1$ is a set of subsets of~$\G_0$, instead of the set of edges as for a usual graph.
We will be considering oriented hypergraphs for which $\G_1$ is a set of arrows completed by a set
of oriented $2$-path.

\begin{defn}
\label{AllGraphDef}
Given a quiver $\Qc$  with no loops and no $2$-cycles,
an {\it extended quiver} $\widetilde{\Qc}$ with {\it underlying} quiver $\Qc$, 
is an oriented hypergraph defined as follows.
\begin{enumerate}
\item[{\bf (A)}]
The vertices of $\Qc$ are labeled by $\{x_1,\ldots,x_n\}$,
$\widetilde{\Qc}$ has $m$ extra ``colored'' vertices labeled by the odd variables 
$\{\xi_1,\ldots,\xi_m\}$, so that
$$
\widetilde{\Qc}_0={\Qc}_0\cup\{\xi_1,\ldots,\xi_m\}.
$$

\item[{\bf (B)}]
Some of the new vertices $\{\xi_1,\ldots,\xi_m\}$ are related 
by $2$-paths through
the vertices $\{x_1,\ldots,x_n\}$ of the underlying quiver $\Qc$.
The set of arrows $\Qc_1$ is completed by the set of $2$-paths:
$$
\widetilde{\Qc}_1=\Qc_1\cup_{k}\left\{(\xi_i\to{}x_k\to\xi_j)\right\}.
$$

\item[{\bf (C)}] $2$-paths with opposite orientation:
$\xi_i\to{}x_k\to\xi_j$ and $\xi_j\to{}x_k\to\xi_i$ 
are not allowed.

\end{enumerate}
\end{defn}

Although $\widetilde{\Qc}$ is a hypergraph, and therefore can hardly be represented
graphically, the above definition is illustrated by the following diagram
$$
 
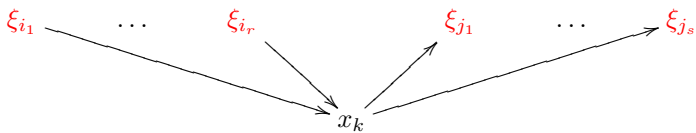
\xymatrix{
{\color{red}\xi_{i_1}}\ar@{->}[rrrd]&\dots&
{\color{red}\xi_{i_r}}\ar@{->}[rd]&&
{\color{red}\xi_{j_1}}\ar@{<-}[ld]&\dots&
{\color{red}\xi_{j_s}}\ar@{<-}[llld]\\
&&& x_k&&
}
$$
representing a certain collection of $2$-paths (less or equal to $r\times{}s$ paths between the odd variables).

\begin{rem}
Note that since all odd vertices are frozen in the current approach  we do not consider arrows between the odd vertices of $\widetilde{\Qc}$,
and this is certainly an interesting question whether one can add such arrows and create a more rich
combinatorics of extended quivers.
\end{rem}

\subsection{Extended quiver mutations}\label{GraphMut}

Let us define the mutation rules of an extended quiver.
These mutations are performed at even vertices only.

\begin{defn}
\label{MutDef}
Given an extended quiver~$\widetilde{\Qc}$ and an even vertex $x_k\in\Qc_0$,
the mutation~$\widetilde\mu_k$ is defined by the
following rules:
\begin{enumerate}
\item[(0)] 
The underlying quiver $\Qc\subset\widetilde{\Qc}$ mutates
according to the same rules as in the classical case~\cite{FZ1}.

\item[(1*)] 
Given a $2$-path $(\xi_i\to x_k\to\xi_j)\in\widetilde{\Qc}_1$,
for all $x_\ell\in\Qc_0$
connected to~$x_k$ by an outgoing arrow $(x_k\to{}x_\ell)$, 
add the $2$-paths $(\xi_i\to x_\ell\to\xi_j)$.

\item[(2*)]  Reverse all the $2$-paths through $x_k$, i.e.,
change $(\xi_i\to x_k\to\xi_j)$ to $(\xi_i\leftarrow x_k\leftarrow\xi_j)$.

\item[(3*)]  Remove two-by-two the $2$-paths through $x_k$
which are identical but have opposite orientations, eventually created by rule (1*), i.e.,
$2$-paths $(\xi_i\to{}x_\ell\to\xi_j)$ and 
$(\xi_i\leftarrow{}x_\ell\leftarrow\xi_j)$ cancel each other.
\end{enumerate}
\end{defn}

The above rules can be illustrated by the diagram:
$$ 
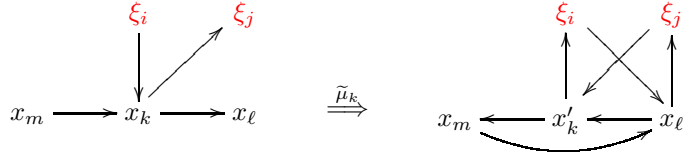
\xymatrix{
&{\color{red}\xi_i}\ar@{->}[d]&
{\color{red}\xi_j}\ar@<-2pt>@{<-}[ld]\\
x_m&x_k\ar@{<-}[l]&x_\ell\ar@{<-}[l]
}
\qquad
\xymatrix{\\
\stackrel{\widetilde\mu_k}{\Longrightarrow}}
\qquad
 \xymatrix{
&{\color{red}\xi_i}\ar@<2pt>@{->}[rd]\ar@{<-}[d]&
{\color{red}\xi_j}\ar@<-2pt>@{->}[ld]\ar@{<-}[d]\\
x_m\ar@/^-1pc/[rr]&x_k'\ar@{->}[l]&x_\ell\ar@{->}[l]
}
$$

\begin{ex}
One has
$$
 
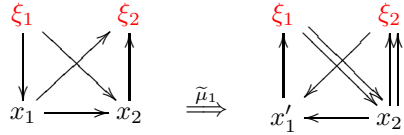
\xymatrix{
{\color{red}\xi_1}\ar@{->}[rd]\ar@{->}[d]&
{\color{red}\xi_2}\ar@{<-}[ld]\ar@{<-}[d]\\
x_1&x_2\ar@{<-}[l]
}
\quad
\xymatrix{\\
\stackrel{\widetilde\mu_1}{\Longrightarrow}}
\quad
 \xymatrix{
{\color{red}\xi_1}\ar@<3pt>@{->}[rd]\ar@{->}[rd]\ar@{<-}[d]&
{\color{red}\xi_2}\ar@{->}[ld]\ar@{<-}[d]\ar@<3pt>@{<-}[d]\\
x_1'&x_2\ar@{->}[l]
}
$$
This mutation creates a new $2$-path $(\xi_1\to{}x_2\to\xi_2)$,
so that the resulting extended quiver has two such paths 
(and not four as appears if one counts arrows between $\xi$'s and $x$'s).
\end{ex}

\subsection{Weighted quivers: the case of  two 
odd vertices}\label{WeSect}

The simplest class of extended quivers are those that have exactly two 
odd vertices,~$\xi_1$ and~$\xi_2$.
Such an extended quiver is equivalent to the usual quiver~$\Qc$
together with a function on the set of vertices 
$$
w:\Qc_0\to\Z
$$
that counts the number of oriented $2$-paths through the vertex.
For every~$x_i\in\Qc_0$, 
\begin{enumerate}
\item
a {\it positive} value of~$w(x_i)$
corresponds to the number of $2$-path $\xi_1\to{}x_i\to\xi_2$;
\item
a {\it negative} value of~$w(x_i)$
corresponds to the number of $2$-path $\xi_2\to{}x_i\to\xi_1$.
\end{enumerate}
\noindent
This function was called the {\it weight function} in~\cite{OT} where it was applied to integer sequences.

The quiver mutations defined above read as follows in terms of the weight function.
The mutation~$\widetilde\mu_k$ at $k$th vertex sends $w$ to
the new function~$\widetilde\mu_k(w)$ defined by:
\begin{equation}
\label{WFMut}
\begin{array}{rcll}
\widetilde\mu_k(w)(i) &=& w(i)+[b_{ki}]_+w(k), & i\not={}k,\\[4pt]
\widetilde\mu_k(w)(k) &=& -w(k),&
\end{array}
\end{equation}
where $[b_{ki}]_+$ is the number of arrows from the vertex
$k$ to the vertex $i$, and if the vertices are oriented from $i$ to $k$, then
$[b_{ki}]_+=0$.

\begin{ex}
\label{ThEx}
The following ``Somos-4 quivers'' (cf.~\cite{FM} and~\cite{Mar}):
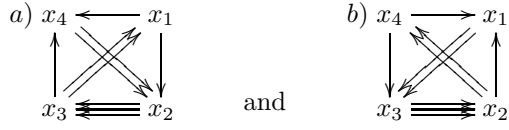
$$
a)  \xymatrix{
x_4\ar@<3pt>@{->}[rd]\ar@{->}[rd]\ar@{<-}[d]&
x_1\ar@<-3pt>@{<-}[ld]\ar@{<-}[ld]\ar@{->}[d]\ar@{->}[l]\\
x_3&x_2\ar@{->}[l]\ar@<-2pt>@{->}[l]\ar@<2pt>@{->}[l]
}
\xymatrix{\\
\qquad\hbox{and}\qquad}
b) \xymatrix{
x_4\ar@<3pt>@{<-}[rd]\ar@{<-}[rd]\ar@{->}[d]&
x_1\ar@<-3pt>@{->}[ld]\ar@{->}[ld]\ar@{<-}[d]\ar@{<-}[l]\\
x_3&x_2\ar@{<-}[l]\ar@<-2pt>@{<-}[l]\ar@<2pt>@{<-}[l]
}
$$
are examples of so-called {\it period~$1$ quivers}.
Each of them performs a cyclic rotation under the mutation at $x_1$, e.g.,
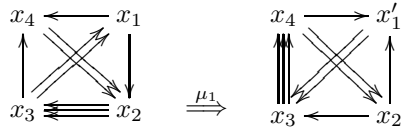
$$
 \xymatrix{
x_4\ar@<3pt>@{->}[rd]\ar@{->}[rd]\ar@{<-}[d]&
x_1\ar@<-3pt>@{<-}[ld]\ar@{<-}[ld]\ar@{->}[d]\ar@{->}[l]\\
x_3&x_2\ar@{->}[l]\ar@<-2pt>@{->}[l]\ar@<2pt>@{->}[l]
}
\xymatrix{\\
\quad
\stackrel{\mu_1}{\Longrightarrow}
\quad}
 \xymatrix{
x_4\ar@<3pt>@{->}[rd]\ar@{->}[rd]\ar@{<-}[d]\ar@<-2pt>@{<-}[d]\ar@<2pt>@{<-}[d]&
x_1'\ar@<-3pt>@{->}[ld]\ar@{->}[ld]\ar@{<-}[d]\ar@{<-}[l]\\
x_3&x_2\ar@{->}[l]
}
$$
In both cases, there exists a period~$1$ weight function:
$$
\begin{array}{ccccc}
a)&w(x_1) = 1,& w(x_2) = 0,& w(x_3) =0,& w(x_4) = -1,\\[2pt]
b)&w(x_1) = 1,& w(x_2) = 1,& w(x_3) =-1,& w(x_4) = -1,
\end{array}
$$
respectively, that also rotates under mutation~(\ref{WFMut}); see~\cite{OT}.
In our initial notation, the above weight functions correspond to the following extended quivers:
$$
a) 
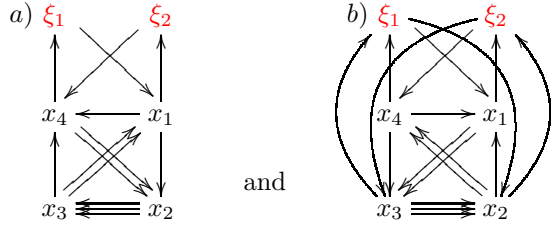
\xymatrix{
{\color{red}\xi_1}\ar@<3pt>@{->}[rd]\ar@{<-}[d]&
{\color{red}\xi_2}\ar@<-2pt>@{->}[ld]\ar@{<-}[d]\\
x_4\ar@<3pt>@{->}[rd]\ar@{->}[rd]\ar@{<-}[d]&
x_1\ar@<-3pt>@{<-}[ld]\ar@{<-}[ld]\ar@{->}[d]\ar@{->}[l]\\
x_3&x_2\ar@{->}[l]\ar@<-2pt>@{->}[l]\ar@<2pt>@{->}[l]
}
\xymatrix{\\
\\
\qquad\hbox{and}\qquad}
b) \xymatrix{
{\color{red}\xi_1}\ar@<3pt>@{->}[rd]\ar@{<-}[d]
\ar@/^2.2pc/[rdd]&
{\color{red}\xi_2}\ar@<-2pt>@{->}[ld]\ar@{<-}[d]\ar@/^-2.2pc/[ldd]\\
x_4\ar@<3pt>@{<-}[rd]\ar@{<-}[rd]\ar@{->}[d]&
x_1\ar@<-3pt>@{->}[ld]\ar@{->}[ld]\ar@{<-}[d]\ar@{<-}[l]\\
x_3\ar@/^1.7pc/[uu]&x_2\ar@/^-1.7pc/[uu]\ar@{<-}[l]\ar@<-2pt>@{<-}[l]\ar@<2pt>@{<-}[l]
}
$$
Note that the period~$1$ weight function is unique up to a multiple.

This leads to a family of extensions of the Somos-4 sequence~\cite{OT}.
 \end{ex}

\subsection{Relation to BGP-functor}\label{WeBGPSect}

Let us mention that formula~(\ref{WFMut})
appeared in several different situations~\cite{GHK,Gal1,Gal2}.
Furthermore, remarkably enough, this formula is nothing else but the dual formula for that of the classical
{\it Bernstein-Gelfand-Ponomarev functor (BGP-functor)} in the theory of quiver representations.

Let $Q$ be a quiver, its vertex $v$ be a source. 
BGP-functor $F_v$ acts on the quivers and their representations in the following way. It takes quiver $Q$ to $\bar Q$ whose vertices and arrows 
are in one-to-one correspondence with vertices and arrows of $Q$, 
for vertex $w$ and arrow $\alpha$ of $Q$ we will denote the corresponding vertex and arrow of $\bar Q$ by
 $\bar w$ and $\bar \alpha$, correspondingly.
$F_v(w)=\bar w$, $F_v(\alpha)=\bar\alpha$ for any $w$ and $\alpha$ of $Q$. 
If edge $\alpha=(w_1\to w_2)$, where $w_1\ne v$ and $w_2\ne v$, then $F_v(\alpha)=\bar\alpha=(\bar w_1\to\bar w_2)$.
Edge $\alpha=(v\to w)$ becomes $\bar\alpha=(\bar w\to\bar v)$. Hence, vertex $\bar v$ is a sink in $\bar Q$.

If $R$ is a representation of quiver $Q$, i.e., any vertex $u$ of $Q$ corresponds to a vector space~$R(u)$ 
and an arrow $\alpha=(s\to t)$ corresponds to a linear map $R(\alpha):R(s)\to R(t)$,
then, $F_v: Rep(Q)\to Rep(\bar Q)$ is as follows:
$F_v(R(w))=R(w)$ for any $w\ne v$; $F(R(\alpha))=R(\alpha)$ if none of endpoints of $\alpha$ coincides with $v$. 
Finally, $F_v(R(v))$ is defined as follows. 
Let 
$$
{\bf\beta}=\{\beta_i:v\to w_i, i=[1,k]\}
$$ 
be the collection of all arrows from $v$, $R({\bf \beta}):R(v)\to \oplus_{i=1}^k R(w_i)$ is a sum of all maps $R(\beta_i)$, i.e., 
$R(\beta)(x)=(R(\beta_1)(x),\dots,R(\beta_k)(x))$ for any $x\in R(v)$. 
Define 
$$
\Im(R({\bf \beta}))=\{{\bf \beta}(x), x\in R(v)\}=\{\oplus_i\beta_i(x)\in\oplus_{i=1}^k R(w_i),\ x\in R(v)\},
$$ 
and set $F_v(R(v))=\oplus_{i=1}^k R(w_i)/\Im(R({\bf\beta}))$.
Assuming a nondegeneracy property: $\Im(R({\bf\beta}))\simeq R(v)$ we see that dimension vector 
$(\dim_R(w))_{w\in Q}$, $\dim_R(w)=\dim(R(w))$ for all $w\in Q$ of representation~$R$ changes as follows.  
$$
\begin{array}{rcl}
\overline{\dim}_{F_v(R)}(\bar w) &=& \dim_R(w),\qquad w\ne v,\\[4pt]
\overline{\dim}_{F_v(R)}(\bar v) &= & \sum_{i=1}^k\dim_R(w_i)-\dim_R(v).
\end{array}
$$
Note that formula~(\ref{WFMut}) describes the dual transformation.

Similarly, formula~(\ref{WFMut}) describes the dual transformation to the change of dimension vector if $v$ is a sink as well.

\section{Exchange relations}\label{Echangist}

We define the exchange relations of the even variables
$\{x_1,\ldots,x_n\}$
corresponding to mutations of extended quivers defined in the previous section.

\subsection{Introducing exchange relations}\label{EchangistSub}

\begin{defn}
\label{ExchDef}
Given an extended quiver $\widetilde{\Qc}$,
the mutation $\widetilde\mu_k$ replaces the variable $x_k$ by a new variable, $x_k'$,
other variables remain unchanged:
$$
\widetilde\mu_k:\{x_1,\ldots,x_n,\xi_1,\ldots,\xi_m\}\to
\{x_1,\ldots,x_n,\xi_1,\ldots,\xi_m\}\setminus\{x_k\}\cup\{x_k'\}.
$$
The new variable is defined by the following formula
\begin{equation}
\label{Mute}
x_kx_k'=
\prod\limits_{\substack{x_k\to x_\ell }}x_\ell
\quad+\quad
\prod\limits_{\substack{\xi_i\to{}x_k\to\xi_j}}(1+\xi_i\xi_j)
\prod\limits_{\substack{x_\ell\to x_k }}x_\ell,
\end{equation}
that will be called, as in the classical case, an exchange relation.

We denote by $A(\widetilde\Qc)$ the associative commutative
superalgebra generated by the initial variables $\{x_1,\ldots,x_n,\xi_1,\ldots,\xi_m\}$
together with all the mutations of $x_k$.
\end{defn}

Note that,
after substitution $\xi\equiv0$, the above formula obviously coincides with the
exchange relations for the classical cluster algebra corresponding to the underlying quiver $\Qc$.
The first summand in (\ref{Mute}) is exactly as in the classical case, the second
one is modified.

\begin{rem}
\label{SqMutLem}
Note that, unlike the classical case, the above mutation of $x_k$ is
{\it not an involution}.
\end{rem}

\subsection{Example: quivers with weight function}
In the simplest case of two colored vertices,
that we considered in Section~\ref{WeSect},
formula~(\ref{Mute}) reads
$$
x_kx_k'=
\prod\limits_{\substack{x_k\to x_\ell }}x_\ell
\quad+\quad
(1+w_k\varepsilon)
\prod\limits_{\substack{x_\ell\to x_k }}x_\ell,
$$
where $\varepsilon:=\xi_1\xi_2$ denotes the product of the odd variables (cf.~\cite{OT}).
Indeed, one obviously has $(1+\varepsilon)^{w_k}=(1+w_k\varepsilon)$ since $\varepsilon^2=0$.

\subsection{Example: the supergroup $\OSp(1|2)$}\label{ElExSec}
One of the first examples of cluster algebras
given in~\cite{FZ1} is the algebra of regular functions on the Lie group $\SSL(2)$.
We consider here its superanalog.

The most elementary superanalog of the group $\SSL(2)$ is the supergroup $\OSp(1|2)$.
Let $\Rc=\Rc_0\oplus\Rc_1$ be a commutative ring.
The set of $\Rc$-points of the supergroup $\OSp(1|2)$  is the following $3|2$-dimensional supergroup of matrices:
\begin{equation}
\label{OSpRel}
\left(
\begin{array}{cc|c}
a&b&\g\\[4pt]
c&d&\d\\[4pt]
\hline
\a&\b&e
\end{array}
\right)
\qquad
\hbox{such that}
\qquad
\begin{array}{rcl}
ad&=&1+bc-\a\b,\\[4pt]
e&=&1+\a\b,\\[4pt]
\g&=&a\b-b\a\\[4pt]
\d&=&c\b-d\a.
\end{array}
\end{equation}
The elements $a,b,c,d,e\in\Rc_0$, and $\a,\b,\g,\d\in\Rc_1$;
these elements are generators of the algebra of regular functions on $\OSp(1|2)$.

Choose the initial cluster coordinates $(a,b,c,\a,\b)$,
and consider the following quiver:
$$
 
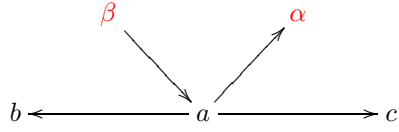
\xymatrix{
&{\color{red}\b}\ar@{->}[rd]&&
{\color{red}\a}\ar@{<-}[ld]\\
b\ar@{<-}[rr]&& a\ar@{->}[rr]&&c
}
$$
The coordinate $d$ is then the mutation of $a$, i.e., $a'=d$.
Indeed, the exchange relation~\eqref{Mute} for the coordinate $a$ reads
$$
aa'=1+bc+\b\a,
$$
which is precisely the first equation for $\OSp(1|2)$ relating $a$ and $d$.
Note that, similarly to the $\SL_2$-case 
the coordinates $b$ and $c$ are frozen cf.~\cite{FZ1}.

\section{Laurent phenomenon and invariant presymplectic form}\label{PrSect}

In this section we discuss two general properties of the constructed algebras, namely
the Laurent phenomenon and  invariant presymplectic form.

\subsection{The Laurent phenomenon}\label{PrLPSect}
Since the division by odd coordinates is not well-defined, 
all the Laurent polynomials we consider have denominators
equal to some monomials in $\{x_1,\ldots,x_n\}$.

Our first statement is the following.

\begin{thm}
\label{LeurThm}
For every extended quiver~$\widetilde{\Qc}$, all the 
rational functions $x_k', x_k'',\ldots$,
obtained recurrently by any series of consecutive admissible mutations,
 are Laurent polynomials
in the initial coordinates $\{x_1,\ldots,x_n,\xi_1,\ldots,\xi_m\}$.
\end{thm}

\begin{proof}
This statement follows from the classical Laurent phenomenon~\cite{FZ2},
after the identification of Section~\ref{MTSect}.
\end{proof}

\subsection{The presymplectic form}\label{ISFSect}
Consider the following differential $2$-form:
\begin{equation}
\label{SupOM}
\displaystyle
\om=
\sum\limits_{\substack{x_i\rightarrow x_j }}
\frac{dx_i\wedge{}dx_j}{x_ix_j}+
\sum\limits_{\substack{\xi_i\to{}x_\ell\to\xi_j}}
\frac{d\left(
\xi_i\xi_j
\right)\wedge
dx_\ell}{x_\ell}.
\end{equation}
Note that the summation goes over the elements of $\widetilde\Qc_1$.
The first summand is nothing but the standard presymplectic form
(see~\cite{GSV1,GSV}) associated to the cluster algebra
corresponding to the underlying quiver $\Qc$.

\begin{thm}
\label{SymThm}
For every extended quiver~$\widetilde{\Qc}$, the form $\om$
is invariant under mutations $\widetilde\mu_k$, combined with the exchange relations~(\ref{Mute}).
\end{thm}

In other words, expressing $x_k$ in terms of the other variables and $x_k'$,
and substituting to~(\ref{SupOM}),
one obtains precisely the presymplectic form associated to the extended quiver 
$\widetilde\mu_k(\widetilde{\Qc})$.

\begin{proof}  Note that $d(\xi_i\xi_j)=d(1+\xi_i\xi_j)$. Then, the statement follows from the standard result of invariance of compatible presymplectic form under a cluster mutation.
\end{proof}

\section{Reduction to the classical case: the monomial transform}\label{MTSect}

In this section we explain how to reduce the mutation rule of Section~\ref{GraphMut}
and the exchange relations~(\ref{Mute}) to the usual mutations combined with a coordinate transformation
that, following~\cite{GSV}, we call the {\it monomial transform}.
This in particular will imply the Laurent phenomenon formulated in Section~\ref{PrLPSect}.

\subsection{Monomial transform: the definition}

Given a quiver $\Qc$ (with neither loops nor $2$-cycles),
assume that the vertices of $\Qc$ are split into two groups.
To fix the notation, we set:
$$
\Qc_0=
\left\{x_1,\ldots,x_{n_1}\right\}
\cup
\left\{y_1,\ldots,y_{n_2}\right\}.
$$
In other words, $\Qc$ is a colored quiver (with two colors).

The {\it monomial transform} at $x_k$, that we denote by $T_k$, consists of three steps:
\begin{enumerate}
\item
add a new arrow $(x_m\to{}y_i)$ whenever $x_m$ and $y_i$ are connected to $x_k$ by ingoing arrows;
\item
add a new arrow $(y_i\to{}x_\ell)$ whenever $y_i$ is connected to $x_k$ by an ingoing arrow and $x_\ell$
is connected to $x_k$ by an outgoing arrow;
\item
change the variable $x_k$ to
\begin{equation}
\label{MonoM}
\widetilde x_k=
x_k/y_i.
\end{equation}
\end{enumerate}

The monomial transform is illustrated by the following diagram:
$$ 
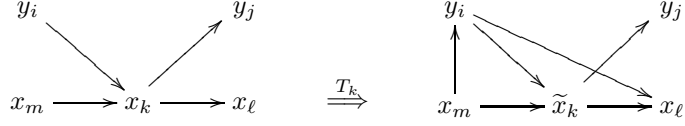
\xymatrix{
{y_i}\ar@{->}[rd]&&
{y_j}\ar@<-2pt>@{<-}[ld]\\
x_m&x_k\ar@{<-}[l]&x_\ell\ar@{<-}[l]
}
\xymatrix{\\
\qquad
\stackrel{T_k}{\Longrightarrow}
\qquad}
\xymatrix{
{y_i}\ar@{->}[rd]\ar@{->}[rrd]&&
{y_j}\ar@<-2pt>@{<-}[ld]\\
x_m\ar@{->}[u]&\widetilde x_k\ar@{<-}[l]&x_\ell\ar@{<-}[l]
}
$$

\subsection{From extended quiver to colored quiver}
Given an extended quiver $\widetilde{\Qc}$ (cf. Section~\ref{DeFCluSS}),
let us construct a colored quiver $\Qc^c$ according to the following rule.
\begin{enumerate}
\item
If $\widetilde{\Qc}_0=\{x_1,\ldots,x_n,\xi_1,\ldots,\xi_m\}$,
we set  
\begin{equation}
\label{IdEq}
{\Qc}_0^c=
\left\{
x_1,\ldots,x_n,y_{ij},\;\vert\;1\leq i<j\leq m
\right\};
\qquad
y_{ij}:=1+\xi_i\xi_j.
\end{equation}

\item
For every oriented $2$-path $(\xi_i\to x_k\to\xi_j)$, add an arrow $(y_{ij}\to x_k)$.
\end{enumerate}
Conversely, given a colored quiver $\Qc^c$ with
${\Qc}_0^c=
\left\{
x_1,\ldots,x_n,y_{ij},\;\vert\;1\leq i<j\leq m
\right\}$,
one reconstructs an extended quiver $\widetilde{\Qc}$ with $\widetilde{\Qc}_0=\{x_1,\ldots,x_n,\xi_1,\ldots,\xi_m\}$.
We denote by $\mathcal I$ the above identification between the extended quivers and the chosen class of colored quivers.

The corresponding diagram is:
$$ 
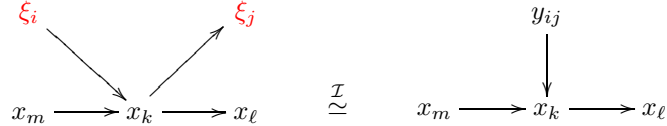
\xymatrix{
{\color{red}\xi_i}\ar@{->}[rd]&&
{\color{red}\xi_j}\ar@<-2pt>@{<-}[ld]\\
x_m&x_k\ar@{<-}[l]&x_\ell\ar@{<-}[l]
}
\xymatrix{\\
\qquad
\stackrel{\mathcal I}{\simeq}
\qquad}
 \xymatrix{
&y_{ij}\ar@{->}[d]&\\
x_m&x_k\ar@{<-}[l]&x_\ell\ar@{<-}[l]
}
$$

\subsection{Mutations composed with monomial transforms}
It turns out that the mutations of extended quivers and the exchange relations
described in Sections~\ref{GraphMut} and~\ref{EchangistSub}
are nothing but the usual mutations of the corresponding colored quivers
composed with the monomial transform.

\begin{prop}
\label{RedProp}
One has
$\widetilde\mu_k=\mathcal{I}\circ\mu_k\circ\mathcal{I}$.
\end{prop}

\begin{proof}
Straightforward from~(\ref{IdEq}).
\end{proof}

\section{the main example: superfriezes}\label{SFS}

Frieze patterns were invented by Coxeter~\cite{Cox}.
This notion provides surprising relations between classical continued fractions,
projective geometry (cross-ratios) and quiver representations.
Coxeter's friezes are also related to linear difference equations and
the classical moduli spaces $\mathcal{M}_{0,n}$ of configurations of points,
see~\cite{SVRS}.
The set of Coxeter's friezes is an algebraic variety
that has a structure of cluster varieties,
associated to the Dynkin quivers $A_n$.
For a survey, see~\cite{Mor}.

The notion of {\it superfrieze} was introduced in~\cite{MOT}
as generalization of Coxeter's frieze patterns.
The collection of all superfriezes is an algebraic supervariety
isomorphic to the supervariety of supersymmetric Hill's 
(or one-dimensional Schr\"odinger) equations~\cite{SVRS}
with some particular monodromy condition.

In this section we describe how the extended quiver~$\widetilde{\Qc}$ with the underlying quiver~$A_n$, 
correspondis to a superfrieze.

\subsection{Supersymmetric discrete Schr\"odinger equation}
Consider two infinite sequences of elements of some supercommutative ring $\Rc$:
$$
(a_i),
\quad
(\b_i),
\quad
i\in\Z,
$$
where $a_i\in\Rc_0$ and $\b_i\in\Rc_1$.

The following equation with indeterminate $(V_i,W_i)_{i\in\Z}$:
\begin{equation}
\label{SeQE}
\left(
\begin{array}{l}
V_{i-1}\\[4pt]
V_i\\[4pt]
W_i
\end{array}
\right)=
A_i\left(
\begin{array}{l}
V_{i-2}\\[4pt]
V_{i-1}\\[4pt]
W_{i-1}
\end{array}
\right),
\qquad
\hbox{where}
\qquad
A_i=
\left(
\begin{array}{cc|c}
0&1&0\\[4pt]
-1&a_i&-\b_i\\[4pt]
\hline
0&\b_i&1
\end{array}
\right),
\end{equation}
is the supersymmetric version of discrete Schr\"odinger equation, see~\cite{MOT}.
Note that the matrix $A_i$ belongs to the supergroup $\OSp(1|2)$.

We assume that the coefficients $a_i,\b_i$ are 
(anti)periodic with some period $n$:
$$
a_{i+n}=a_i,
\qquad
\b_{i+n}=-b_i,
$$
for all $i\in\Z$.
Under this assumption, there is a notion of {\it monodromy},
i.e., an element $M\in\OSp(1|2)$, such that periodicity properties of
the solutions of~(\ref{SeQE}) are described by $M$.

Supersymmetric discrete Schr\"odinger equations
with fixed monodromy matrix:
\begin{equation}
\label{MoQE}
M=\left(
\begin{array}{rr|c}
-1&0&\;\;0\\[4pt]
0&-1&\;\;0\\[4pt]
\hline
0&0&\;\;1
\end{array}
\right).
\end{equation}
were considered in~\cite{MOT}.
This is an algebraic supervariety of dimension $n|(n+1)$
which is a version of super moduli space $\mathfrak{M}_{0,n}$, see~\cite{Wit}.
The notion of superfrieze allows one to define special coordinates on this supervariety.

\subsection{The definition of a superfrieze and the corresponding superalgebra}\label{TheDef}

Similarly to the case of classical Coxeter's friezes, a superfrieze
is a horizontally-infinite array bounded by rows of $0$'s and $1$'s.
Even and odd elements alternate and form ``elementary diamonds'';
there are twice more odd elements.

\begin{defn}
A superfrieze, or a supersymmetric frieze pattern, is the following array
$$
\begin{array}{ccccccccccccccccccccccccc}
&\ldots&0&&&&0&&&&0\\[10pt]
\ldots&{\color{red}0}&&{\color{red}0}&&{\color{red}0}
&&{\color{red}0}&&{\color{red}0}&&\ldots\\[10pt]
\;\;\;1&&&&1&&&&1&&&\ldots\\[10pt]
&{\color{red}\varphi_{0,0}}&&{\color{red}\varphi_{\frac{1}{2},\frac{1}{2}}}&&{\color{red}\varphi_{1,1}}
&&{\color{red}\varphi_{\frac{3}{2},\frac{3}{2}}}&&{\color{red}\varphi_{2,2}}&&\ldots\\[12pt]
&&f_{0,0}&&&&f_{1,1}&&&&f_{2,2}\\[10pt]
&{\color{red}\varphi_{-\frac{1}{2},\frac{1}{2}}}&&{\color{red}\varphi_{0,1}}
&&{\color{red}\varphi_{\frac{1}{2},\frac{3}{2}}}
&&{\color{red}\varphi_{1,2}}&&{\color{red}\varphi_{\frac{3}{2},\frac{5}{2}}}&&\ldots\\[10pt]
f_{-1,0}&&&&f_{0,1}&&&&f_{1,2}&&\\[4pt]
&\iddots&&\iddots&& \ddots&&\ddots&& \ddots&&\!\!\!\ddots\\[4pt]
&&f_{2-m,1}&&&&f_{0,m-1}&&&&f_{1,m}&&&&\\[10pt]
\ldots&{\color{red}\varphi_{\frac{3}{2}-m,\frac{3}{2}}}&&{\color{red}\varphi_{2-m,2}}&&\ldots
&&{\color{red}\varphi_{0,m}}&&{\color{red}\varphi_{\frac{1}{2},m+\frac{1}{2}}}&&{\color{red}\varphi_{1,m+1}}\\[10pt]
\;\;\;1&&&&1&&&&1&&&&&\\[10pt]
\ldots&{\color{red}0}&&{\color{red}0}&&{\color{red}0}
&&{\color{red}0}&&{\color{red}0}&&{\color{red}0}&\\[10pt]
&\ldots&0&&&&0&&&&0&\ldots
\end{array}
$$
where $f_{i,j}\in\Rc_0$ and $\varphi_{i,j}\in\Rc_1$, and where every 
{\it elementary diamond}:
$$
\begin{array}{ccccc}
&&B&&\\[4pt]
&{\color{red}\Xi}&&{\color{red}\Psi}&\\[4pt]
A&&&&D\\[4pt]
&{\color{red}\Phi}&&{\color{red}\Sigma}&\\[4pt]
&&C&&
\end{array}
$$
satisfies the following conditions:
\begin{equation}
\label{Rule}
\begin{array}{rcl}
AD-BC&=&1+\Sigma\Xi,\\[4pt]
B\Phi-A\Psi&=&\Xi,\\[4pt]
B\Sigma-D\Xi&=&\Psi,
\end{array}
\end{equation}
that we call the {\it frieze rule}.

The integer $m$, i.e., the number of even rows between the rows
of $1$'s is called the {\it width} of the superfrieze.
\end{defn}

The last two equations of~(\ref{Rule}) are equivalent to
$$
A\Sigma-C\Xi=\Phi,
\qquad
D\Phi-C\Psi=\Sigma.
$$
Note also that these equations also imply $\Xi\Sigma=\Phi\Psi$,
so that the first equation can also be written as follows: 
$AD-BC=1-\Phi\Psi$.

One can associate an elementary diamond with every element of $\OSp(1|2)$
using the following formula:
$$
\left(
\begin{array}{cc|c}
a&b&\g\\[4pt]
c&d&\d\\[4pt]
\hline
\a&\b&e
\end{array}
\right)
\qquad
\longleftrightarrow
\qquad
\begin{array}{ccccc}
&&\!\!\!-a&&\\[4pt]
&{\color{red}\g}&&{\color{red}\a}&\\[4pt]
b&&&&\!\!\!\!-c\\[4pt]
&\!\!\!{\color{red}-\b}&&{\color{red}\d}&\\[4pt]
&&d&&
\end{array}
$$
so that the relations~(\ref{OSpRel}) and~(\ref{Rule}) coincide.

Consider also the configuration:
$$
\begin{array}{ccccccc}
&&{\color{red}\widetilde{\Psi}}&&{\color{red}\widetilde{\Xi}}\\[4pt]
&&&B&&\\[4pt]
{\color{red}\widetilde{\Phi}}&&{\color{red}\Xi}&&{\color{red}\Psi}&&{\color{red}\widetilde{\Sigma}}\\[4pt]
&A&&&&D\\[4pt]
&&{\color{red}\Phi}&&{\color{red}\Sigma}&\\[4pt]
&&&C&&
\end{array}
$$
The frieze rule~\eqref{Rule} then implies
$$
B\,(\Phi-\widetilde{\Phi})=
A\,(\Psi-\widetilde{\Psi}),
\qquad
B\,(\Sigma-\widetilde{\Sigma})=
D\,(\Xi-\widetilde{\Xi}).
$$

\begin{defn}
The supercommutative superalgebra generated by
all the entries of a superfrieze will
be called the {\it algebra of a superfrieze}.
\end{defn}

\subsection{Examples: superfriezes of width $1$ and $2$}

The most general superfrieze of width $m=1$ is of the following form:
$$
\begin{array}{ccccccccccccccccccccccc}
&&0&&&&0&&&&\!\!0&&&&0\\[8pt]
&{\color{red}0}&&{\color{red}0}&&{\color{red}0}&&
\!\!\!{\color{red}0}&&{\color{red}0}&&{\color{red}0}&&\;\;\;{\color{red}0}&&\!\!\!{\color{red}0}\\[8pt]
1&&&&1&&&&1&&&&\!\!1&&&&1\\[8pt]
&{\color{red}\xi}&&{\color{red}\xi}&&{\color{red}\xi'}
&&{\color{red}\xi'}&&{\color{red}\xi-x\eta}&&{\color{red}\xi-x\eta}
&&{\color{red}\eta}&&{\color{red}\eta}\\[10pt]
&&x&&&&x'&&&&\!\!\!x&&&&x'\\[10pt]
&{\color{red}\xi-x\eta}&&\;\;\;{\color{red}x\eta-\xi}&&\;{\color{red}\eta}&&{\color{red}-\eta}
&&{\color{red}-\xi}&&{\color{red}\xi}&&{\color{red}-\xi'}&&{\color{red}\xi'}\\[8pt]
1&&&&1&&&&1&&&&\!\!1&&&&1\\[8pt]
&{\color{red}0}&&{\color{red}0}&&\;\;{\color{red}0}
&&\!{\color{red}0}&&{\color{red}0}&&{\color{red}0}&&\;\;\;{\color{red}0}&&\!\!\!{\color{red}0}\\[10pt]
&&0&&&&0&&&&\!\!0&&&&0\
\end{array}
$$
where
$$
x'=\frac{2}{x}+\frac{\eta\xi}{x},
\qquad 
\xi'=\eta-\frac{2\xi}{x}.
$$
One can chose local coordinates $(x,\xi,\eta)$
to parametrize the supervariety of superfriezes.

The next example is a superanalog of
so-called Gauss' ``Pentagramma mirificum'':
$$
\begin{array}{cccccccccccccccccccccc}
0&&&&\!\!\!0&&&&\!\!0&&&&\!\!\!0&&&&0&&\ldots\\[10pt]
&{\color{red}0}&&\!\!{\color{red}0}&&{\color{red}0}
&&{\color{red}0}&&{\color{red}0}&&{\color{red}0}&&\!\!{\color{red}0}&&{\color{red}0}&&{\color{red}0}\\[10pt]
\ldots&&1&&&&1&&&&1&&&&\!\!1&&&&\!\!1\\[4pt]
&{\color{red}\xi^*}
&&\!\!{\color{red}\xi}&&{\color{red}\xi}&&{\color{red}\xi'}&&{\color{red}\xi'}
&&{\color{red}\raisebox{.5pt}{\textcircled{\raisebox{-.9pt} {1}}}}
&&{\color{red}\raisebox{.5pt}{\textcircled{\raisebox{-.9pt} {1}}}}
&&\!\!{\color{red}\zeta^*}&&{\color{red}\zeta^*}
\\[10pt]
y'&&&&
\!\!x&&&&\!\!x'
&&&&x''&&&&y\\[10pt]
&\!\!\!{\color{red}-\eta'}&&
\!\!\!{\color{red}\eta^*}
&&{\color{red}\raisebox{.5pt}{\textcircled{\raisebox{-.9pt} {2}}}}
&&{\color{red}\eta}
&&{\color{red}\raisebox{.5pt}{\textcircled{\raisebox{-.9pt} {2}}}'}
&&{\color{red}\eta'}
&&\!\!\!{\color{red}\eta^*}&&\!\!{\color{red}-\eta}
&&{\color{red}\eta}\\[10pt]
&&x''&&&&y
&&&&y'&&&&x
&&&&\!\!\!x'\\[10pt]
&{\color{red}\raisebox{.5pt}{\textcircled{\raisebox{-.9pt} {1}}}}
&&{\color{red}-\raisebox{.5pt}{\textcircled{\raisebox{-.9pt} {1}}}}
&&{\color{red}\zeta^*}
&&{\color{red}-\zeta^*}
&&{\color{red}\zeta}
&&\!\!{\color{red}-\zeta}&&\!\!\!{\color{red}\zeta'}
&&{\color{red}-\zeta'}&&{\color{red}-\xi'}\\[10pt]
1&&&&\!\!1&&&&1&&&&1&&&&1&&\ldots\\[10pt]
&{\color{red}0}&&\!\!{\color{red}0}&&{\color{red}0}&&{\color{red}0}&&{\color{red}0}
&&{\color{red}0}&&\!\!{\color{red}0}&&{\color{red}0}&&{\color{red}0}\\[10pt]
\ldots&&0&&&&0&&&&0&&&&0&&&&\!\!0
\end{array}
$$
The frieze is defined by the initial values
$(x,y,\xi,\eta,\zeta)$, the next values are easily calculated using the frieze rule:
$$
x'=\frac{1+y}{x}+\frac{\eta\xi}{x},
\qquad
y'=\frac{1+x+y}{xy}+\frac{\eta\xi}{xy}+\frac{\zeta\eta}{y}.
$$
One then calculates:
$$
x''=\frac{1+y'}{x'}+\frac{\eta'\xi'}{x'}
=\frac{1+x}{y}+\frac{\eta\xi}{y}+\xi\zeta+\frac{x}{y}\zeta\eta.
$$
All these Laurent polynomials
can be obtained as mutations of the
initial coordinates $(x,y,\xi,\eta,\zeta)$ and the initial quiver
$$
 
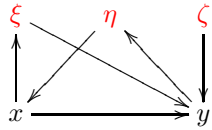
\xymatrix{
{\color{red}\xi}\ar@{<-}[d]\ar@{->}[rrd]&
{\color{red}\eta}\ar@{->}[ld]\ar@{<-}[rd]&
{\color{red}\zeta}\ar@{->}[d]\\
x\ar@{->}[rr]&&y
}
$$

For the odd coordinates, one has:
$$
\xi'=\eta-x'\xi=\eta-\frac{1+y}{x}\xi,
\qquad
\eta'=\zeta-y'\xi=\zeta-\frac{1+x+y}{xy}\xi-\frac{\xi\eta\zeta}{y},
\qquad
\zeta'=-\xi
$$
On the other side of the initial diagonal,
$$
\zeta^*=\eta-y\zeta,
\qquad
\eta^*=\xi-x\zeta,
\qquad
\xi^*=-\zeta.
$$
Furthermore,
$$
\raisebox{.5pt}{\textcircled{\raisebox{-.9pt} {1}}}=
\frac{(1+x)}{y}\eta-\xi-\zeta,
\qquad
\raisebox{.5pt}{\textcircled{\raisebox{-.9pt} {2}}}=
x\eta-y\xi,
$$
and finally:
$$
\raisebox{.5pt}{\textcircled{\raisebox{-.9pt} {2}}}'=
x'\zeta-y'\eta
=\frac{1+y}{x}\zeta-\frac{1+x+y}{xy}\eta-\frac{\xi\eta\zeta}{x}.
$$

\subsection{Properties of superfriezes}

The main properties of superfriezes are similar to those of the classical Coxeter friezes, see~\cite{MOT}.

(a)
The property of {\it glide symmetry} reads:
$$
f_{i,j}=f_{j-m-1,i-2},\qquad
\varphi_{i,j}=\varphi_{j-m-\frac{3}{2},i-\frac32},\qquad
\varphi_{i+\frac12,j+\frac12}=-\varphi_{j-m-1,i-1}.
$$
This implies, in particular, the following (anti)periodicity:
$$
\varphi_{i+n,j+n}=-\varphi_{i,j},
\qquad
f_{i+n,j+n}=f_{i,j},
$$ 
for all $i,j\in\Z$.

(b)
The Laurent phenomenon:
entries of a superfrieze are Laurent polynomials in 
the entries from any of its diagonals.

(c)
The collection of all superfriezes of width $n$ is an algebraic supervariety
of superdimension~$n|(n+1)$.
It is isomorphic to the supervariety of
Schr\"odinger equations~(\ref{SeQE}) with monodromy condition~(\ref{MoQE}).
The relation to difference equations is as follows.
The entries of the South-East diagonal of every superfrieze are solutions to
the discrete Schr\"odinger equation~(\ref{SeQE}).

\subsection{Superfriezes viewed as cluster superalgebras}

Let us now describe the cluster structure of the supervariety of superfriezes.
Consider the following quiver with $m$ even and $m+1$ odd vertices:
\begin{equation}
\label{AQuiv}
 \xymatrix{
{\color{red}\xi_1}\ar@{<-}[rd]\ar@{->}[rrrd]&&
{\color{red}\xi_2}\ar@{->}[ld]\ar@{<-}[rd]\ar@{->}[rrrd]&&
{\color{red}\xi_3}\ar@<-3pt>@{->}[ld]\ar@{<-}[rd]
&\cdots&
{\color{red}\xi_{m}}&&
{\color{red}\xi_{m+1}}\ar@{->}[ld]\\
& x_1\ar@{->}[rr]&&
x_2\ar@{->}[rr]
&&x_3
&\cdots&\ar@{->}[lu]x_m
}
\end{equation}
and the corresponding algebra.

\begin{thm}
\label{ClFrProp}
The algebra of a superfrieze of width $m$ 
is a subalgebra of the algebra corresponding to the above quiver.
\end{thm}

\begin{proof}
Choose the following entries of the superfrieze on parallel diagonals:
$$
 \begin{array}{cccccccccccccc}
1&&&&1\\[4pt]
&{\color{red}*}&&{\color{red}\xi_1}&&{\color{red}*}&&{\color{red}\xi'_1}\\[4pt]
&&x_1&&&&x'_1\\[4pt]
&&&{\color{red}*}&&{\color{red}\xi_2}&&{\color{red}*}&&{\color{red}\xi'_2}\\[4pt]
&&&&x_2&&{\color{red}\ddots}&&x'_2&&\!\!{\color{red}\ddots}\\[4pt]
&&&&&\ddots&&{\color{red}\xi_m}&&\ddots&&{\color{red}\xi'_m}\\[4pt]
&&&&&&x_m&&&&x'_m\\[4pt]
&&&&&&&{\color{red}*}&&{\color{red}\xi_{m+1}}&&{\color{red}*}&&{\color{red}\xi'_{m+1}}\\[4pt]
&&&&&&&&1&&&&1
\end{array}
$$
The entries $\{x_1,\ldots,x_m,\xi_1,\ldots,\xi_{m+1}\}$ determine all other entries of the superfrieze, 
and can be taken for initial coordinates.
Our goal is to calculate the entries $\{x'_1,\ldots,x'_m,\xi'_1,\ldots,\xi'_{m+1}\}$
and show that these entries 
also belong to the algebra $A(\widetilde\Qc)$ of the quiver~\eqref{AQuiv}.

Using the frieze rule~\eqref{Rule}, one obtains the following recurrent formula:
\begin{equation}
\label{LFor}
x_kx'_k=1+x_{k+1}x'_{k-1}+\xi_{k+1}\xi_k.
\end{equation}
On the other hand, let us perform consecutive mutations at vertices 
$x_1$, and then at $x_2,x_3\ldots,x_m$
of the quiver~\eqref{AQuiv}.
After the $(k-1)$st step, one obtains the following quiver:
$$
 
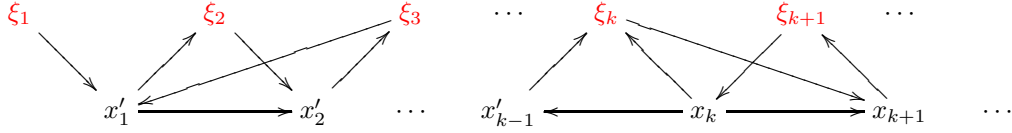
\xymatrix
  @!0 @R=1.3cm @C=1.3cm
  {
{\color{red}\xi_1}\ar@{->}[rd]&&
{\color{red}\xi_2}\ar@{<-}[ld]\ar@{->}[rd]&&
{\color{red}\xi_3}\ar@<-1pt>@{<-}[ld]\ar@{->}[llld]
&\cdots&
{\color{red}\xi_{k}}\ar@{<-}[ld]\ar@{->}[rrrd]&&
{\color{red}\xi_{k+1}}\ar@{->}[ld]\ar@{<-}[rd]&\cdots\\
& x'_1\ar@{->}[rr]&&
x'_2
&\cdots&x'_{k-1}&&
\ar@{->}[lu]x_k\ar@{->}[rr]\ar@{->}[ll]&&x_{k+1}&\cdots
}
$$
Therefore, the mutation at $x_k$ is allowed,
and the exchange relation for $x_k$ is exactly the same as the
recurrent formula \eqref{LFor} for $x'_k$.
We have proved that the values of the entries $\{x'_1,\ldots,x'_m\}$ in the frieze coincide
with the coordinates $\{x'_1,\ldots,x'_m\}$ of the quiver~\eqref{AQuiv} after 
the iteration of even mutations.

Note that after $m$ consecutive mutations at even vertices, the quiver~\eqref{AQuiv}
becomes as follows:
$$
 
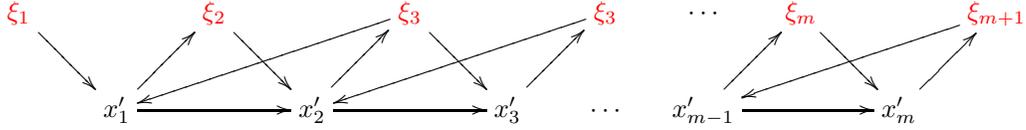
\xymatrix
   @!0 @R=1.3cm @C=1.3cm
 {
{\color{red}\xi_1}\ar@{->}[rd]&&
{\color{red}\xi_2}\ar@{<-}[ld]\ar@{->}[rd]&&
{\color{red}\xi_3}\ar@<-1pt>@{<-}[ld]\ar@{->}[rd]\ar@{->}[llld]
&&{\color{red}\xi_3}\ar@<-1pt>@{<-}[ld]\ar@{->}[llld]&\cdots&
{\color{red}\xi_{m}}\ar@{<-}[ld]&&
{\color{red}\xi_{m+1}}\ar@{<-}[ld]\ar@{->}[llld]\\
& x'_1\ar@{->}[rr]&&
x'_2\ar@{->}[rr]
&&x'_3
&\cdots&x'_{m-1}\ar@{->}[rr]&&\ar@{<-}[lu]x'_m
}
$$

Consider now the odd entries of the superfrieze $\{\xi'_1,\ldots,\xi'_{m+1}\}$,
and let us proceed by induction.

For the first of the odd entries, one has:
$$
\xi'_1=\xi_2-x'_1\xi_1.
$$
Indeed, the frieze rule implies that the entry between 
$\xi_1$ and $\xi'_1$ (previously denoted by $*$) is also equal to $\xi'_1$,
i.e., we have the following fragment of the superfrieze:
$$
\begin{array}{cccccccccc}
&&&&1&&&&1\\[4pt]
&&&{\color{red}\xi_1}&&{\color{red}\xi'_1}&&{\color{red}\xi'_1}\\[4pt]
&&x_1&&&&x'_1\\[4pt]
&&&{\color{red}*}&&{\color{red}\xi_2}
\end{array}
$$
The above expression for $\xi'_1$ is just the third equality in~\eqref{Rule}.
It follows that $\xi'_1$ belongs to the algebra $A(\widetilde\Qc)$.

It was proved in~\cite{MOT} that the entries on the diagonals
of the superfrieze satisfy recurrence equations with coefficients
standing in the first two rows.
In particular, Lemma 2.5.3 of~\cite{MOT} implies the
following recurrence for the odd entries of the superfrieze:
$$
\xi'_k-\xi'_{k-1}=
-\xi_1x'_k,
\qquad
\hbox{for all}
\quad
k.
$$
One concludes, by induction on $k$,
that all of the entries $\{\xi'_1,\ldots,\xi'_{m+1}\}$ 
belong to the algebra $A(\widetilde\Qc)$.
Again, using the induction one arrives at the same conclusion for all
parallel diagonals.

Finally, one proves in a similar way
that the entries in-between, denoted by $*$,
also belong to the algebra $A(\widetilde\Qc)$.
\end{proof}

\bigskip
{\bf Acknowledgements}.
We are grateful to Sophie Morier-Genoud,
Gregg Musiker and Sergei Tabachnikov 
for a number of fruitful discussions.

\end{document}